\documentclass[12pt]{amsart}

\usepackage{amsthm,amssymb,xcolor,graphicx,caption}
\usepackage[hidelinks]{hyperref}
\usepackage[top=1.2in,left=1in,right=1in,bottom=1.5in]{geometry}
\usepackage{fancyhdr}

\theoremstyle{plain}
\newtheorem{theorem}{Theorem}

\begin{document}

\title{ON A SQUARE PACKING CONJECTURE OF ERD\H{O}S}
\author{Anshul Raj Singh}
\date{\today}
\maketitle
\vspace{-1cm}
\begin{abstract}
    \noindent Let $f(n)$ be the maximum sum of the sides of non-overlapping squares (or equilateral triangles) packed inside a unit square or (unit equilateral triangle). In this paper, we explore some properties of $f$ and examine how the square and triangle cases are similar. We prove that a conjecture of Erd\H{o}s, which says that $f(k^2+1) = k$ for all $k$, is equivalent to the convergence of the series $\sum_{k\geqslant 1}(f(k^2+1)-k)$. We also explore the case of parallelograms and discuss how that is similar to the case of unit square and triangle.
\end{abstract}
    
\section{Introduction}Let $f(n)$ denote the maximum sum of the side lengths of $n$ non-overlapping squares packed inside a unit square. A simple application of Cauchy-Schwarz inequality implies that $f(n^2) = n$ for all $n$. Erd\H{o}s \cite{Er94} conjectured that $f(n^2+1) = n$ for all positive integers $n$, but this still remains open.

In this paper we will first look at this conjecture for the case of unit equilateral triangle (unit side length) which turns out to be very similar to the case of unit square. With a slightly modified definition of $f$ we will also see that the same properties hold for the case of parallelogram.

\section{Main results}
\begin{theorem}
    Let $f:\mathbb{N}\rightarrow\mathbb{R}$ such that for all $a\leqslant b$, we have $$af(m)\leqslant a^2 - b^2 + bf(b^2 - a^2 + m)\text{ and } f(m^2+1)\geqslant m\quad (\ast)$$for any positive integer $m$. Let $\epsilon(k) := f(k^2+1) - k$, then
    \begin{enumerate}
        \item If $\epsilon(n) = 0$ for some $n$, then $\epsilon(k)$ for all $k\leqslant n$.
        \item If $\epsilon(n) > 0$ for some $n$, then $\epsilon(k) = \Omega(1/k)$, i.e., there exists a positive constant $c$ and an integer $k_0$ such that $\epsilon(k)\geqslant c/k$ for all $k\geqslant k_0$.
    \end{enumerate}
\end{theorem}

\begin{proof}
    If $n$ is such that $\epsilon(n) = 0$, we want to show that $\epsilon(k) = 0$ for all $k\leqslant n$. Let $n\geqslant k$ be such that $\epsilon(n) = 0$. Set $m = k^2+1, a = k$, and $b = n$. Putting these values in $(\ast)$, we get
    $$kf(k^2+1)\leqslant k^2-n^2+nf(n^2-k^2+k^2+1)\Rightarrow kf(k^2+1)\leqslant k^2.$$
    Hence $f(k^2+1) = k$ for all $k\leqslant n$. This proves $(1)$.

    Now let $n$ be such that $f(n^2+1) = n+\alpha$ with $\alpha > 0$. Let $a = n$, $m = n^2+1$ and $b\geqslant n$, therefore, by $(\ast)$ we get
    $$nf(n^2+1) = n^2 + n\alpha \leqslant n^2 - b^2 + bf(b^2 + 1) = n^2 + b\epsilon(b).$$
    Therefore, $\epsilon(b)\geqslant n\alpha/b$ for all $b\geqslant n$. This proves $(2)$.
\end{proof}

\section{The Case of triangle}
Let $a_1, a_2, \ldots, a_n$ be the side lengths of non-pverlapping equilateral triangles packed inside a unit equilateral triangle (side length $1$). Let $f(n) = \max\{a_1+\cdots+a_n\}$ over all the possible packings of these $n$ triangles. Note that by comparing the areas, we get $\sqrt{3}(a_1^2+\cdots+a_n^2)/4\leqslant \sqrt{3}/4\implies a_1^2+\cdots+a_n^2\leqslant 1$. On the other hand, by the Cauchy-Schwarz inequality, we have
$$n \geqslant (a_1^2+\cdots+a_n^2)(n) \geqslant (a_1+\cdots a_n)^2 \implies f(n) \leqslant \sqrt{n}.$$
Since an equilateral triangle can always be packed with a square number of congruent equilateral triangles, we obtain that $f(k^2) = k$ for all positive integers $k$. A natural question, similar to Erd\H{o}s' conjecture, arises: Can we say that $f(k^2+1) = k$ for all $k$? 

\begin{theorem}
    $f(k^2+1) = k$ for all $k$ if and only if $\sum_{k\geqslant 1}\epsilon(k)$ converges.
\end{theorem}

\noindent First, we claim that $f$ satisfies $(\ast)$. Note that $f(k^2+1)\geqslant k$, see Fig.1 for an example.

\begin{center}
    \includegraphics[width=50mm]{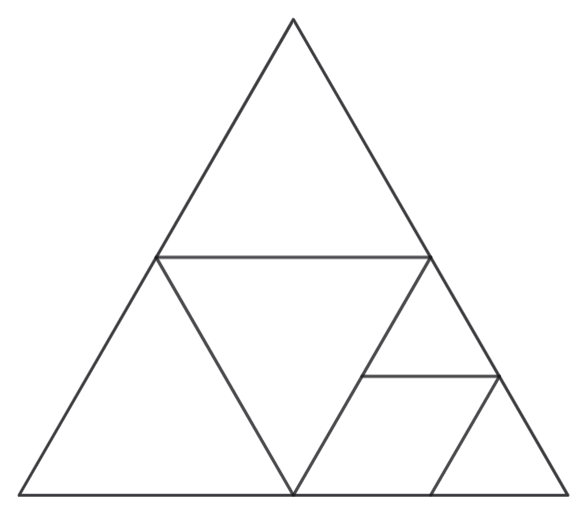}\\
    Fig.1, example: $n=2$, side sum $=2$
\end{center}

\noindent Now to prove the other inequality, we use a similar argument as Iwan Praton \cite{Pr08}.

Take a unit equilateral triangle and divide it into $b\times b$ standard triangular grid, see Fig.2, each triangle having side length $1/b$. Remove, from the bottom right corner, an $a\times a$ subgrid and replace it with $n$ non-overlapping triangles in optimal configuration, i.e., maximizing the side sum. We now have $n$ squares shrunk by a factor of $a/b$ thus having side sum $af(n)/b$ and $b^2-a^2$ squares of side length $1/b$. Therefore, the total sum of sides of triangles is $(b^2-a^2)/b + af(n)/b$. Also, we have $b^2-a^2+n$ triangles in total. Thus, by the definition of $f$, we obtain $$b-\frac{a^2}{b} + a\frac{f(n)}{b}\leqslant f(b^2-a^2+n)\implies af(n)\leqslant a^2 - b^2 + bf(b^2-a^2+n).$$This shows that $f$ satisfies $(\ast)$.
\vspace{0.3cm}

\begin{center}
    \includegraphics[width=50mm]{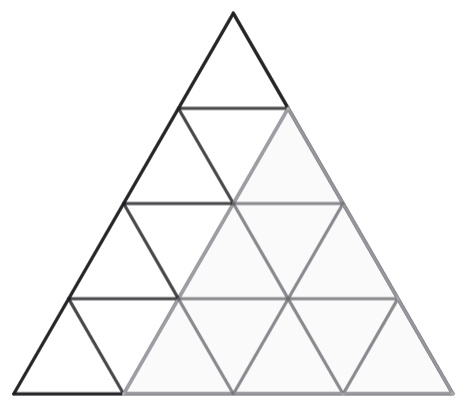}\\
    Fig.2, example: $b = 4$, $a = 3$
\end{center}
\vspace{0.3cm}

Note that this argument works for any shape that can be tiled with a square number of congruent shapes similar to it.

Now by \textit{Theorem 1}, we obtain that if $f(n^2+1) = n$ for some $n$ then $f(k^2+1) = k$ for all $k\leqslant n$. In particular, if $f(k^2+1) = k$ for infinitely many $k$ then $f(k^2+1) = k$ for all $k$. Now we prove \textit{Theorem 2}.

\begin{proof}[Proof of Theorem 2]
    The forward direction is trivial. For the converse part, assume that the series converges, then by part $(2)$ of \textit{Theorem 1}, we cannot have an $n$ such that $\epsilon(n) > 0$ otherwise $\epsilon(k) = \Omega(1/k)\implies \sum_{k\geqslant 1}\epsilon(k)$ diverges.
\end{proof}

\section{The Case of Unit Square (Erd\H{o}s)}
Let $a_1, a_2, \ldots, a_n$ be the side lengths of non-overlapping squares packed inside a unit square. Let $f(n) = \max\{a_1+\cdots+a_n\}$. Then $f$ satisfies $(\ast)$. This is also proven by the same argument as in the case of the equilateral triangle.

Therefore, we obtain the same result, i.e., Erd\H{o}s' conjecture holds if and only if the sum $\sum_{k\geqslant 1}\epsilon(k)$ converges. Additionally, we have that if $f(k^2+1) = k$ for infinitely many $k$, then the conjecture is true.

\section{The Case of parallelograms}
It is clear that we have to modify $f$ for this case. Without loss of generality, let one side of the parallelogram be $1$ and the other be $x$. Let the sides of the similar parallelograms covering this one be $a_i$ and $a_ix$ for $i = 1, 2, \ldots, n$. Then, $f(n) := \max\{a_1+\cdots+a_n\}$. It is easy to see that since this parallelogram can be tiled with a square number of congruent parallelograms similar to it, this $f$ satisfies $(\ast)$. Therefore, the results from the previous case hold for this case as well.

\end{document}